\documentclass{article}
\usepackage{graphicx} 
\usepackage{amsmath}
\usepackage{amssymb}
\usepackage{amsthm}
\usepackage{mathrsfs}
\usepackage{CJKutf8}
\usepackage[utf8]{inputenc}
\usepackage{diagbox}
\usepackage[figuresright]{rotating} 
\usepackage[section]{placeins}
\usepackage{makecell}
\usepackage{longtable}
\allowdisplaybreaks[4]
\usepackage{xcolor}
\usepackage{embrac}
\usepackage{xpatch}
\usepackage{ytableau}
\usepackage{tikz}
\usetikzlibrary{positioning}

\newcommand{\sgn}{\operatorname{sgn}}

\newtheorem{theorem}{Theorem}[section]

\newtheorem{lemma}[theorem]{Lemma}
\newtheorem{proposition}[theorem]{Proposition}
\newtheorem{corollary}[theorem]{Corollary}

\newtheorem{conjecture}{Conjecture}[section]

\def\de{\begin{equation}}
	\def\ee{\end{equation}}

\title{Upper bound of some character ratios and large genus asymptotic behavior of Hurwitz numbers}
\author{Xiang Li\footnote{Email: lxiang1993@ustc.edu.cn.}}
\begin{document}
	\maketitle
	\begin{CJK}{UTF8}{gbsn}
		\begin{abstract}
			In \cite{L2} we found the large genus asymptotics of Hurwitz numbers for the Riemann sphere with a fixed number of general profiles and some $(2,1^{d-2})$ profiles. In this paper, motivated from \cite{DLLY}, we generalize these results to Hurwitz numbers of an arbitrary compact Riemann surface with a fixed number of general profiles and some $(r,1^{d-r})$ profiles.
		\end{abstract}
		\section{Introduction}
		The notion of Hurwitz numbers was introduced in \cite{Hur1,Hur2}. The connected Hurtwitz numbers (for short Hurwitz numbers), denoted by $H^X_{g,d}(\theta^{(1)}, \dots , \theta^{(n)})$, are the number of connected ramified coverings $f: C\rightarrow X$ of degree $d$, where $C$ is a Riemann surface of genus $g$, $X$ is a Riemann surface of genus $g(X)$, and the ramification profiles over $n$ marked points are given by the partitions $\theta^{(1)},\dots,\theta^{(n)}\vdash d$.
		
		For a partition $\theta=(\theta_1,\dots,\theta_{l})\vdash d$ with $\theta_1\geq\theta_2\geq\cdots\geq\theta_{l}> 0$, denote $|\theta|=\sum_{i=1}^{l}\theta_{i}=d,l(\theta)=l$, $l^*(\theta)=d-l$ and $z_\theta=\prod\limits_i m_i(\theta)! i^{m_i}$ with $m_i(\theta)$ being the multiplicity of $i$ in $\theta$.
		
		Hurwitz \cite{Hur1,Hur2} gave a closed formula for $H_{g,d}(2\,1^{d-2},2\,1^{d-2},\dots)=:H_{g,d}$ with fixed $d$, which implies their
		structure with fixed $d$ and their large genus asymptotics; these were briefly reviewed in~\cite{DYZ}. 
		
		In \cite{DYZ} Dubrovin-Yang-Zagier obtained a simple recursion for $H_{g,d}$ based on the Pandharipande equation, and used it to give a new proof of the structure of these numbers with fixed $d$ and the large genus asymptotics. We \cite{L1} generalized this to double Hurwitz numbers (i.e., $H_{g,d}(\mu^{(1)},\mu^{(2)},2\,1^{d-2},2\,1^{d-2},\allowbreak\dots)$) by deriving Pandharipande-type equations. 
		
		Motivated by the work of Ding-Li-Liu-Yan \cite{DLLY} (cf.~\cite{DLLY2}), we are towards extending the above-mentioned
		results by replacing $2\,1^{d-2}$ by $r\,1^{d-r}$ (and by more general $\nu$) in this paper. This leads to the problem of finding the irreducible representations that yield the largest and second-largest character ratios for a fixed conjugacy class $\mu$. The largest character ratio is already known \cite{FOW} (see Theorem~B below).		
		For the case $\mu=(r,\,1^{d-r})$, Frumkin-James-Roichman provided a combinatorial interpretation of the central character $f_{(r,\,1^{d-r})}(\lambda)$ (see \eqref{f} for the definition) as a signed count of Young trees of order $r$ contained in the diagram $\lambda$ \cite{FJR}.
		
		We will prove in Section~\ref{proof} the following
		
		\begin{theorem}\label{cH2}
			For any fixed $d\geq7,2\leq r\leq d-2,s\geq0$ and $\mu^{(1)},\dots,\mu^{(s)}\vdash d$, we have
			\begin{align}
				&H^X_{g,d}(\mu^{(1)},\dots,\mu^{(s)},r\,1^{d-r},r\,1^{d-r},\dots)\nonumber\\ 
				=&\frac{2d!^{2g(X)}}{d!^2}\prod_{i=1}^s\frac{d!}{z_{\mu^{(i)}}} \sum_{1\leq m\leq \frac{d!}{r(d-r)!}}b^X_r(\mu^{(1)},\dots,\mu^{(s)},m) m^{\frac{1}{r-1}(2g+(2-2g(X))d-2-\sum_{i=1}^s l^*(\mu^{(i)}))}, \label{cHv}
			\end{align}
			where $d!^{2g(X)}\prod_{i=1}^s\frac{d!}{z_{\mu^{(i)}}}b_r^{X}(\mu^{(1)},\dots,\mu^{(s)},m)$ are integers with
			
			1. $b^X_r(\mu^{(1)},\dots,\mu^{(s)},\frac{d!}{r(d-r)!})=1$;			
			
			2. $b^X_r(\mu^{(1)},\,\cdots,\,\mu^{(s)},m)=0$ for $\frac{(d-1)!}{r\cdot(d-r-1)!}<m<\frac{d!}{r(d-r)!}$;
			
			3. $b^X_r(\mu^{(1)},\,\cdots,\,\mu^{(s)},\frac{(d-1)!}{r\cdot(d-r-1)!})=-d^{2-2g(X)-s}\prod_{i=1}^{s}m_1(\mu^{(i)})$;
			
			4. $b^X_r(\mu^{(1)},\dots,\mu^{(s)},m)=0$ for $\frac{(d-r-1)d!}{r(d-1)(d-r)!}<m<\frac{(d-1)!}{r\cdot(d-r-1)!}$;
			
			5. $b_{r}^{X}(\mu^{(1)},\,\cdots,\,\mu^{(s)},\frac{(d-r-1)d!}{r(d-1)(d-r)!})=(d-1)^{2-2g(X)-s}\prod_{i=1}^{s}(m_1(\mu^{(i)})-1)$.		
		\end{theorem}	
		
		For the case $g(X)=0$, $s=0$ and $r=2$, Theorem~\ref{cH2} can be deduced from \cite{Hur1,Hur2} (cf. also \cite{DYZ}). For the case $g(X)=0$, $s=1$ and $r=2$, Do-He-Robertson \cite{DHR} proved the statement 1 in Theorem \ref{cH2}, and conjectured the statement 2 which was later proved by Yang \cite{Y}. 
		Do-He-Robertson \cite{DHR} also deduced the statement 1 for $H_{g,d}(\mu^{(1)},k^\frac{d}{k},2\,1^{d-2},2\,1^{d-2},\dots)$.
		Some general results are also obtained in \cite{ALR}. We also refer to \cite{DLLY,DLLY2} for other interesting results.
		
		Furthermore, we generalize the statement 1 of Theorem~\ref{cH2} to Theorem~\ref{cH1}, by replacing $r\,1^{d-r}$ with $\nu$. 
		
		\begin{theorem}\label{cH1}
			For any fixed $d\geq5,s\geq0$ and $\mu^{(1)},\dots,\mu^{(s)},\nu\vdash d$, we have
			\begin{align}
				&H^X_{g,d}(\mu^{(1)},\dots,\mu^{(s)},\nu,\nu,\dots)\nonumber\\ =&\frac{2d!^{2g(X)}}{d!^2}\prod_{i=1}^s\frac{d!}{z_{\mu^{(i)}}} \sum_{1\leq m\leq \frac{d!}{z_\nu}}b^X_\nu(\mu^{(1)},\dots,\mu^{(s)},m) m^{(2g+(2-2g(X))d-\sum_{i=1}^s l^*(\mu^{(i)})-2)/l^*(\nu)}, \label{cHv1}
			\end{align}
			where $d!^{2g(X)}\prod_{i=1}^s\frac{d!}{z_{\mu^{(i)}}}b_\nu^{X}(\mu^{(1)},\dots,\mu^{(s)},m)$ are integers with $b^X_\nu(\mu^{(1)},\dots,\mu^{(s)},\frac{d!}{z_\nu})=1$.
		\end{theorem}

		The following corollary easily follows from Theorem~\ref{cH1}.
		\begin{corollary}\label{As1}
			For any fixed $d\geq5,s\geq0$ and $\mu^{(1)},\dots,\mu^{(s)}\vdash d$, the asymptotic behavior of $ H^X_{g,d}(\mu^{(1)},\dots,\mu^{(s)},\nu,\nu,\dots)$ is given by
			\begin{align}
				H^X_{g,d}(\mu^{(1)},\dots,\mu^{(s)},&\nu,\nu,\dots) \sim\frac{2 d!^{s+2g(X)-2}}{z_{\mu^{(1)}} \cdots z_{\mu^{(s)}}}\nonumber\\
				&\times\big(\frac{d!}{z_\nu}\big)^{(2g+(2-2g(X))d-2-\sum_{i=1}^s l^*(\mu^{(i)}))/l^*(\nu)},\,\,g\rightarrow \infty. 
			\end{align}
		\end{corollary}

		This paper is organized as follows: In Sec.~\ref{proof}, we prove Theorem~\ref{cH2} and Theorem~\ref{cH1}. Further remarks are given in Sec.~\ref{Conclusion}.

		\section{Proof of Theorem~\ref{cH2} and Theorem~\ref{cH1}}\label{proof}		
		Our proofs rely on a combinatorial interpretation of the central character $f_{(r,\,1^{d-r})}(\lambda)$ due to Frumkin-James-Roichman. Recall that for partitions $\lambda,\nu\vdash d$, the central character $f_{\mu}(\lambda)$ is defined by
		\begin{align}
			f_{\mu}(\lambda):=\frac{d!}{z_{\mu}}\frac{\chi_\lambda(\mu)}{\text{dim}\lambda}.\label{f}
		\end{align}
		\vspace{3pt}
		\noindent\bf{Theorem~A}\mdseries{} (Frumkin-James-Roichman \cite{FJR}). \textit{For $r\leq d$, $\lambda\vdash d$,
			\begin{align}
				f_{(r,\,1^{d-r})}(\lambda)=\sum_{\substack{\Gamma\in\{\text{Young trees of}\\\text{order $r$ in diagram $\lambda$}\}}} (-1)^{\text{vert}(\Gamma)}\cdot \text{weight}(\Gamma).   \label{fr}
			\end{align}
			A Young tree $\Gamma$ of order $r$ is a connected, acyclic graph on $r$ vertices corresponding to distinct boxes of diagram $\lambda$, in which an edge exists between two vertices if and only if they are in the same row or column with no other vertex of $\Gamma$ lying strictly between them. A simple path in $\Gamma$ is a connected subgraph contained entirely within one row or one column. Such a path is called maximal if it is not strictly contained in any other simple path of $\Gamma$ as a subgraph. The term $\text{vert}(\Gamma)$ counts the vertical edges in $\Gamma$, and $$\text{weight}(\Gamma):=\prod_{p } l(p)!$$ is the product taken over all maximal simple paths $p$ in $\Gamma$, and $l(p)$ being the number of edges in $p$.}
		\vspace{3pt}
		Using this result, we derive the following lemma.

		Based on the definition (cf.~\cite{Hur1,Hur2}) and the Lemma \ref{rm2}, we determine the structure of Hurwitz numbers.

		\begin{lemma}\label{l1}
			\begin{align}
				\frac{|\chi_\lambda(r,\,1^{d-r})|}{\chi_\lambda(1^d)}\leq \frac{1}{r-1}+\frac{r-2}{(r-1)}\Big(\sum_i^{l(\lambda)}\frac{\binom{\lambda_i}{r}}{\binom{d}{r}}+\sum_i^{l(\lambda')}\frac{\binom{\lambda'_i}{r}}{\binom{d}{r}}\Big).\label{7}
			\end{align}			
		\end{lemma}
		\begin{proof}
			For any Young tree $\Gamma$ lying in diagram $\lambda$, it is obvious that
			\begin{align}
				\text{weight}(\Gamma)\leq (r-2)!
			\end{align}
			unless $\Gamma$ lies in a single row or column; in this case we call $\Gamma$ a straight Young tree. As there are at most $\binom{d}{r}$ Young trees in diagram $\lambda$, 
			\begin{align}
				&|f_{(r,\,1^{d-r})}(\lambda)|\leq\sum_{\substack{\Gamma \text{ is a straight}\\ \text{ Young tree}}}  \text{weight}(\Gamma)+\sum_{\substack{\Gamma \text{ is not a straight}\\ \text{ Young tree}}}\text{weight}(\Gamma)\nonumber\\
				&\leq \Big(\sum_i^{l(\lambda)}\binom{\lambda_i}{r}+\sum_i^{l(\lambda')}\binom{\lambda'_i}{r}\Big)\cdot (r-1)!+\Big(\binom{d}{r}-\sum_i^{l(\lambda)}\binom{\lambda_i}{r}-\sum_i^{l(\lambda')}\binom{\lambda'_i}{r}\Big)\cdot(r-2)!.\label{4}
			\end{align}
			By \eqref{f}, the result follows.
		\end{proof}

		\begin{lemma}\label{rm2}
			If $d\geq7, 2\leq r\leq d$, partitions $\lambda\neq(d),(1^d),\vdash d$, then
			\begin{align}
				\frac{|\chi_\lambda(r,\,1^{d-r})|}{\chi_\lambda(1^d)}\leq&\frac{|d-r-1|}{d-1} ,\qquad r\neq d-1,\label{r}\\
				\frac{|\chi_\lambda(d-1,1)|}{\chi_\lambda(1^d)}\leq&\frac{2}{d(d-3)},\label{dMinus1}
			\end{align}
			where equality in \eqref{r} occurs iff $\lambda=(d-1,1)$ or $(2,\,1^{d-2})$ and equality in \eqref{dMinus1} occurs iff $\lambda=(d-2,2)$ or $(2,2,1^{d-4})$.
		\end{lemma}
		\begin{proof}
			By symmetry, without loss of generality, we assume $\lambda_1\geq\lambda'_1$. 
			
			\textbf{a. Case} $\lambda_1\geq\lambda'_1\geq4$ and $5\leq r\leq d-3$. 
			
			We proceed evaluate the number of straight Young tree in diagram $\lambda$. For the rows $i\geq2$, we move all the boxes of row $i$ that lie in columns $j\geq2$, to the right end of the first row, proceeding row by row. For any $\lambda\vdash d$, we definite
			\begin{align}
				\phi_{\lambda} : \text{ \text{boxes in} diagram }\lambda&\longmapsto\text{ \text{boxes in} diagram }(d+1-l(\lambda),\,1^{l(\lambda)-1})\nonumber\\
				(i,j) &\longmapsto \left\{
				\begin{aligned}
					&(i,j)&i\text{ or }j=1\\
					&(1,\sum\nolimits_{k<i}\lambda_k-i+j)      &\text{otherwise }      
				\end{aligned}
				\right.,
			\end{align}
			where $(i,j)$ represents the box at row $i$, column $j$. 
			\begin{lemma}
				The number of straight Young tree of $\lambda$ does not decrease under map $\phi_{\lambda}$.									
			\end{lemma}
			\begin{proof}We consider the three possible types of a straight Young tree $\Gamma\subset \lambda$.
				
				\textbf{i.} $\Gamma$ is contained in a single column. 
				
				If $\Gamma$ lies in the first column, then $\phi(\Gamma)=\Gamma$. If $\Gamma$ is in a column $j$ with $j\geq2$, then $\phi(\Gamma)$ contained in the first row and forms a straight Young tree in $(d+1-l(\lambda),\,1^{l(\lambda)-1})$.
				\\
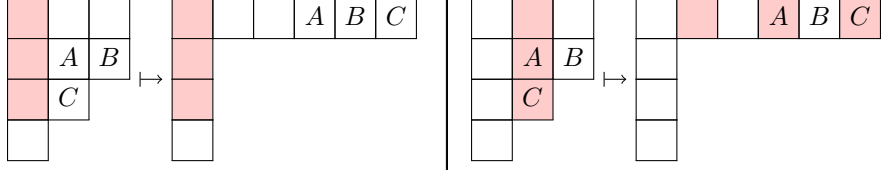
				\begin{minipage}[t]{0.5\textwidth}
					\centering
					\begin{tikzpicture}[node distance=1.5cm]
						\node(A1)at (0,0){
							\begin{ytableau}
								*(red!20)&~&~\\
								*(red!20)&A&B\\
								*(red!20)&C\\
								~
							\end{ytableau}						
						};
						\node(A2)at (3,0){
							\begin{ytableau}
								*(red!20)&~&~&A&B&C\\
								*(red!20)\\
								*(red!20)\\
								~
							\end{ytableau}	
						};
						\draw[|->, thin](A1.east)--(A2.west);
					\end{tikzpicture}	
				\end{minipage}
				\vrule
				\vrule
				\begin{minipage}[t]{0.5\textwidth}
					\centering
					\begin{tikzpicture}[node distance=1.5cm]
						\node(B1)at (0,0){
							\begin{ytableau}
								~&*(red!20)&~\\
								~&*(red!20)A&B\\
								~&*(red!20)C\\
								~
							\end{ytableau}				
						};
						\node(B2)at (3,0){
							\begin{ytableau}
								~&*(red!20)&~&*(red!20)A&B&*(red!20)C\\
								~\\
								~\\
								~
							\end{ytableau}	
						};
						\draw[|->, thin](B1.east)--(B2.west);
					\end{tikzpicture}	
				\end{minipage}
				
				\textbf{ii.} $\Gamma$ is contained in row $i$. And if $i\geq2$, it has no box at position $(i,1)$.
				
				If $\Gamma$ lies in the first row, then $\phi(\Gamma)=\Gamma$. If $\Gamma$ is in a row $i$ with $i\geq2$ and has no box at position $(i,1)$, then $\phi(\Gamma)$ is contained in the first row and forms a straight Young tree in $(d+1-l(\lambda),\,1^{l(\lambda)-1})$.
				\\
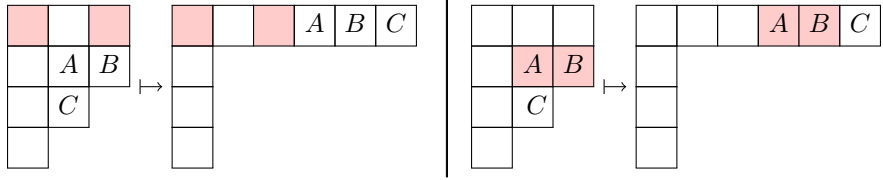
				\begin{minipage}[t]{0.5\textwidth}
					\centering
					\begin{tikzpicture}[node distance=1.5cm]
						\node(A1)at (0,0){
							\begin{ytableau}
								*(red!20)&~&*(red!20)\\
								~&A&B\\
								~&C\\
								~
							\end{ytableau}						
						};
						\node(A2)at (3,0){
							\begin{ytableau}
								*(red!20)&~&*(red!20)&A&B&C\\
								~\\
								~\\
								~
							\end{ytableau}	
						};
						\draw[|->, thin](A1.east)--(A2.west);
					\end{tikzpicture}	
				\end{minipage}
				\vrule
				\vrule
				\begin{minipage}[t]{0.5\textwidth}
					\centering
					\begin{tikzpicture}[node distance=1.5cm]
						\node(B1)at (0,0){
							\begin{ytableau}
								~&~&~\\
								~&*(red!20)A&*(red!20)B\\
								~&~C\\
								~
							\end{ytableau}				
						};
						\node(B2)at (3,0){
							\begin{ytableau}
								~&~&~&*(red!20)A&*(red!20)B&C\\
								~\\
								~\\
								~
							\end{ytableau}	
						};
						\draw[|->, thin](B1.east)--(B2.west);
					\end{tikzpicture}	
				\end{minipage}
				
				\textbf{iii.} $\Gamma$ is contained in row $i$ with $i\geq2$ and has a box at position $(i,1)$.
				
				The $\phi(\Gamma)$ is not a straight Young tree. However, consider the set $\phi(\Gamma\setminus (i,1))\cup (1,1)$ which lies in the first row. After restoring the necessary edge according the definition of a Young tree, it becomes a straight Young tree in $(d+1-l(\lambda),\,1^{l(\lambda)-1})$.
				\\
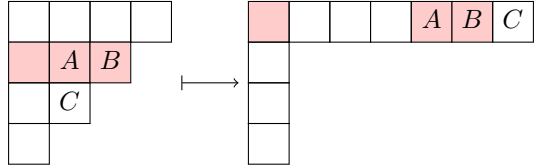
				\begin{minipage}[t]{1\textwidth}
					\centering
					\begin{tikzpicture}[node distance=1.5cm]
						\node(C1)at (0,0){
							\begin{ytableau}
								~&~&~&~\\
								*(red!20)&*(red!20)A&*(red!20)B\\
								~&C\\
								~
							\end{ytableau}						
						};
						\node(C2)at (4,0){
							\begin{ytableau}
								*(red!20)&~&~&~&*(red!20)A&*(red!20)B&C\\
								~\\
								~\\
								~
							\end{ytableau}	
						};
						\draw[|->, thin](C1.east)--(C2.west);
					\end{tikzpicture}	
				\end{minipage}
				In each case, ever $\Gamma$ gives rise to a straight Young tree in $(d+1-l(\lambda),\,1^{l(\lambda)-1})$. Thus, the total number cannot decrease.
			\end{proof}
			
			Since $\lambda_1\geq\lambda'_1\geq4$ and $5\leq r\leq d-3$, the number of straight Young trees is at most $\binom{d-3}{r}$, in the resulting diagram $(d+1-l(\lambda),\,1^{l(\lambda)-1})$. By Lemma~\ref{l1}
			\begin{align}
				\frac{|\chi_\lambda(r,\,1^{d-r})|}{\chi_\lambda(1^d)}\leq\frac{1}{r-1}+\frac{(r-2)(d-r)(d-r-1)(d-r-2)}{(r-1)d(d-1)(d-2)}.\label{1}
			\end{align}
			Notice that
			\begin{align}
				&\frac{d-r-1}{d-1}-\Big(\frac{1}{r-1}+\frac{(r-2)(d-r)(d-r-1)(d-r-2)}{(r-1)d(d-1)(d-2)}\Big)\nonumber\\
				=&\frac{r}{(r-1)d(d-1)(d-2)}((2 r-5) d^2 +(-3 r^2+2 r+10)d+r^3+r^2-4 r-4).\label{2}
			\end{align}
			For the knowledge of quadratic functions and $5\leq r\leq d-3$, we have
			\begin{align}
				(2 r-5) d^2 +(-3 r^2+2 r+10)d+r^3+r^2-4 r-4>0.\label{3}
			\end{align}
			
			\textbf{b. Case} $\lambda_1\leq d-3,\,\lambda'_1\leq4$ and $5\leq r\leq d-3$. The number of straight Young trees is
			\begin{align}
				\binom{\lambda_1}{r}+\binom{\lambda_2}{r}+\binom{\lambda_3}{r}+\binom{\lambda_4}{r}\leq \binom{d-3}{r}.
			\end{align}
			By \eqref{1}, \eqref{2} and \eqref{3}, the result follows.
			
			\textbf{c. Case} $\lambda_1\geq d-2$ and $5\leq r\leq d-3$. By Theorem~A and \eqref{f}, the result follows.
			
			\textbf{d. Case} $r=d-2$. The only diagrams that admit a Young tree are: (1). the diagrams formed by the first row and the first column only; (2) the diagrams that additionally include the boxes at the positions $(2,2)$ or $(2,2),\,(2,3)$ or $(2,2),\,(3,2)$. For $d\geq7$, by Theorem~A and \eqref{f}, the result follows.
			
			\textbf{e. Case} $r=d-1$. The only diagrams that admit a Young tree are: (1). the diagrams formed by the first row and the first column only; and (2) the diagrams that additionally include the box at the position $(2,2)$. By Theorem~A and \eqref{f}, the result follows.
			
			\textbf{f. Case} $r=d$. The only diagrams that admit a Young tree are the diagrams formed by the first row and the first column only. By Theorem~A and \eqref{f}, the result follows.
			
			\textbf{g. Case} $r=2,\,3,\,4$. The values of $f_{(r,\,1^{d-r})}(\lambda)$ have explicit formulas \cite{Fro}:
			\begin{align}
				f_{(2,\,1^{d-2})}(\lambda)=&\frac{1}{2}\sum_{i=1}^{s(\lambda)} \big(b_i(b_i+1)-a_i(a_i+1)\big);\nonumber\\
				f_{(3,\,1^{d-3})}(\lambda)=&\frac{1}{6}\sum_{i=1}^{s(\lambda)} \big(b_i(b_i+1)(2 b_i+1)+a_i(a_i+1)(2a_i+1)\big)-\frac{d(d-1)}{2};\nonumber\\
				f_{(4,\,1^{d-4})}(\lambda)=&\frac{1}{4}\sum_{i=1}^{s(\lambda)} \big(b_i^2(b_i+1)^2-a_i^2(a_i+1)^2-(4d-6)(b_i(b_i+1)-a_i(a_i+1))\big).\nonumber
			\end{align}
			where $a_j=\lambda'_j-j,\,b_i=\lambda_i-i$ and $s(\lambda)$ is the number of diagonal boxes in the position $(i,i)$. By \eqref{f}, the result follows.
		\end{proof}

		The character formula for Hurwitz numbers of genus $g(X)$ targets X \cite{Burn,Fro} is
		\begin{equation}\label{Hurwitz1}
			H_d^{X*}(\theta^{(1)}, \dots , \theta^{(n)}) = \sum_{ \lambda\vdash d}(\frac{\dim \lambda}{d!})^{2-2g(X)} \prod_{i=1}^n f_{\theta^{(i)}}(\lambda),
		\end{equation}		
		where $\dim \lambda$ is the dimension of the irreducible representations of the symmetric group $S(d)$ corresponding to $\lambda$.
		\begin{proposition}\label{dH}
			For any fixed $d\geq7,2\leq r\leq d-2,s\geq0$ and $\mu^{(1)},\dots,\mu^{(s)}\vdash d$, when $k(r-1)+\sum_{i=1}^s l^*(\mu^{(i)})=\text{even}$,
			\begin{align}
				H^{X*}_{d}(\mu^{(1)},\dots,\mu^{(s)}&,r\,1^{d-r},r\,1^{d-r},\dots)\nonumber\\ 
				=&\frac{2d!^{2g(X)}}{d!^2}\prod_{i=1}^s\frac{d!}{z_{\mu^{(i)}}} \sum_{1\leq m\leq \frac{d!}{r(d-r)!}}b^{X*}_r(\mu^{(1)},\dots,\mu^{(s)},m) m^{k}, \label{dHv2}
			\end{align}
			where $d!^{2g(X)}\prod_{i=1}^s\frac{d!}{z_{\mu^{(i)}}}b_r^{X*}(\mu^{(1)},\dots,\mu^{(s)},m)$ are integers with 
			
			1. $b_r^{X*}(\mu^{(1)},\,\cdots,\,\mu^{(s)},\frac{d!}{r(d-r)!})=1$;
			
			2. $b_r^{X*}(\mu^{(1)},\,\cdots,\,\mu^{(s)},m)=0$, for $\frac{(d-r-1)d!}{r(d-1)(d-r)!}<m<\frac{d!}{r(d-r)!}$;
			
			3. $b_r^{X*}(\mu^{(1)},\,\cdots,\,\mu^{(s)},\frac{(d-r-1)d!}{r(d-1)(d-r)!})=(d-1)^{2-2g(X)-s}\prod_{i=1}^{s}(m_1(\mu^{(i)})-1)$.				
		\end{proposition}		
		\begin{proof}
			According to \eqref{Hurwitz1}, 
			\begin{align}
				H^{X*}_{d}(\mu^{(1)},\dots,\mu^{(s)}&,\underbrace{r\,1^{d-r},r\,1^{d-r},\dots}_{k})\nonumber\\
				&=\sum_{\lambda\vdash d}\big(\frac{d!}{\dim\lambda}\big)^{s+g(X)-2}(f_{(r\,1^{d-r})}(\lambda))^k\prod_{i=1}^s \frac{\chi_\lambda(\mu^{(i)})}{z_{\mu^{(i)}}} .\label{5}
			\end{align}				
			Since Lemma~\ref{rm2}, \eqref{f} and $\dim\lambda=\chi_\lambda(1^d)$, we have the structure of disconnected Hurwitz numbers as follows: 
			\begin{align}
				H^{X*}_{d}(\mu^{(1)},\dots,\mu^{(s)},\underbrace{r\,1^{d-r},r\,1^{d-r},\dots}_{k})= \frac{2}{d!^2}\prod_{i=1}^s\frac{d!}{z_{\mu^{(i)}}} &\sum_{1\leq m\leq \frac{d!}{r(d-r)!}}b_r^{X*}(\mu^{(1)},\dots,\mu^{(s)},m) m^{k} ,\nonumber
			\end{align}
			where
			\begin{align}
				b^{X*}_r(\mu^{(1)},\dots,\mu^{(s)},m)=\frac{1}{2}\sum_{\substack{\lambda\vdash d\\|f_{(r\,1^{d-r})}(\lambda)|=m}}(\dim(\lambda))^{2-2g(X)}\big(\sgn(f_{(r\,1^{d-r})}(\lambda))\big)^k\prod_{i=1}^{s}\frac{\chi_\lambda(\mu^{(i)})}{\dim(\lambda)}.      \label{top1}
			\end{align}
			By \eqref{f}, we have
			\begin{align}
				d!^{2g(X)}\prod_{i=1}^s\frac{d!}{z_{\mu^{(i)}}}b_r^{X*}(\mu^{(1)},\dots,\mu^{(s)},m)&=\frac{1}{2}\sum_{\substack{\lambda\vdash d\\|f_{(r\,1^{d-r})}(\lambda)|=m}}(\dim(\lambda))^{2}\Big(\frac{d!}{\dim(\lambda)}\Big)^{2g(X)}\nonumber\\
				&\times\big(\sgn(f_{(r\,1^{d-r})}(\lambda))\big)^k\prod_{i=1}^{s}f_{\mu^{(i)}}(\lambda). 
			\end{align}
			
			Notice that the central character $f_\mu(\lambda)$ is an integer \cite[Corollary 2]{FJR} and $\chi_{\lambda}(\mu)=(-1)^{l^*(\mu)}\chi_{\lambda'}(\mu)$ (for details, see \cite[Example 2 of Section 1.7]{Mac}). It follows that the quantity $d!^{2g(X)}\prod_{i=1}^s\frac{d!}{z_{\mu^{(i)}}}b_r^{X*}(\mu^{(1)},\dots,\mu^{(s)},m)$ is an integer.
			
			When $k(r-1)+\sum_{i=1}^s l^*(\mu^{(i)})=\text{even}$, by Theorem~A, we have
			\begin{align}
				|f_{(r,\,1^{d-r})}(d)|=&|f_{(r,\,1^{d-r})}(1^d)|=\frac{d!}{r(d-r)!},
			\end{align}
			which implies $b_r^{X*}(\mu^{(1)},\,\cdots,\,\mu^{(s)},\frac{d!}{r(d-r)!})=1$.

			Similarly, statement 2 and 3 then follow from Lemma~\ref{rm2} and \eqref{top1}.
		\end{proof}			
		
		\begin{proof}[\bf{Proof of Theorem~\ref{cH2}}]
			By the relationship between connected and disconnected Hurwitz numbers, we have
			\begin{align}
				\sum_{g,d}&\sum_{\mu^{(1)},\dots,\mu^{(s)}\vdash d}\frac{1}{q!} H^X_{g,d}(\mu^{(1)},\dots,\mu^{(s)},\nu,\nu,\dots)x^{q} \prod_{i=1}^s p^i_{\mu^{(i)}}\nonumber\\
				&\qquad=\log(\sum_{k,d}\sum_{\mu^{(1)},\dots,\mu^{(s)}\vdash d}\frac{1}{k!} H^{X*}_{d}(\mu^{(1)},\dots,\mu^{(s)},\underbrace{\nu,\nu,\dots}_{k})x^k \prod_{i=1}^s   p^i_{\mu^{(i)}})\label{rlt2}
			\end{align}
			with $q=(2g+2d-2-\sum_{i=1}^s l^*(\mu^{(i)}))/l^*(\nu)$. Expanding the right hand side of \eqref{rlt2} by Taylor series and taking the coefficients of $\prod_{i=1}^s p^i_{\mu^{(i)}} x^q$ on both sides, we get
			\begin{align}
				H^X_{g,d}(\mu^{(1)},&\dots,\mu^{(s)},\nu,\nu,\dots)=H^{X*}_{d}(\mu^{(1)},\dots,\mu^{(s)},\underbrace{\nu,\nu,\dots}_{q})\nonumber\\
				&-\frac{1}{2}\sum_{\substack{d_1,d_2\geq1\\d_1+d_2=d\\k_1,k_2\geq0\\k_1+k_2=q}}\sum_{\substack{\omega^{(1)},\dots,\omega^{(s)},\nu^{(1)}\vdash d_1\\\sigma^{(1)},\dots,\sigma^{(s)},\nu^{(2)}\vdash d_2\\\omega^{(i)}\cup\sigma^{(i)}=\mu^{(i)},\text{ for}\\i\in\{1,\dots,s\},\nu^{(1)}\cup\nu^{(2)}=\nu}}\frac{q!}{k_1!k_2!}H^{X*}_{d_1}(\omega^{(1)},\dots,\omega^{(s)},\underbrace{\nu^{(1)},\nu^{(1)},\dots}_{k_1})\nonumber\\
				&\qquad\qquad\qquad\qquad\qquad \times H^{X*}_{d_2}(\sigma^{(1)},\dots,\sigma^{(s)},\underbrace{\nu^{(2)},\nu^{(2)},\dots}_{k_2})+\cdots.\label{cHdH}
			\end{align}
			Taking $\nu=(r,1^{d-r}),\,2\leq r\leq d-2$, the $\cdots$ term does not contain $m^q$ with $m\geq \frac{(d-r-1)d!}{r(d-1)(d-r)!}$. By Proposition~\ref{dH}, the term $H^{X*}_{d}(\mu^{(1)},\dots,\mu^{(s)},\underbrace{r\,1^{d-r},r\,1^{d-r},\dots}_{q})$ in \eqref{cHdH} has the leading term 
			\begin{align}
				\frac{2d!^{2g(X)}}{d!^2}\prod_{i=1}^s\frac{d!}{z_{\mu^{(i)}}} \big(\frac{d!}{r(d-r)!}\big)^{q}, \nonumber
			\end{align}
			and second leading term
			\begin{align}
				\frac{2d!^{2g(X)}}{d!^2}\prod_{i=1}^s\frac{d!}{z_{\mu^{(i)}}}(d-1)^{2-2g(X)-s}\prod_{i=1}^{s}(m_1(\mu^{(i)})-1) \big(\frac{(d-r-1)d!}{r(d-1)(d-r)!}\big)^{q}. \nonumber
			\end{align}
			Since Proposition~\ref{dH} and the binomial theorem and $d_1,d_2\geq 1,d_1+d_2=d$, the remaining terms in \eqref{cHdH} have the leading term
			\begin{align}
				\frac{2d!^{2g(X)}}{d!^2}\prod_{i=1}^s\frac{d!}{z_{\mu^{(i)}}} \prod_{i=1}^{s}m_1(\mu^{(i)}) \big(\frac{(d-1)!}{r\cdot(d-r-1)!}\big)^{q}. \nonumber
			\end{align}
			and the power of its second leading is less than that of $H^{X*}_{d}(\mu^{(1)},\dots,\mu^{(s)},\underbrace{r\,1^{d-r},r\,1^{d-r},\dots}_{q})$, which give the statements 1-5 of Theorem~\ref{cH2}.

		\end{proof}
		
		We now turn to the preparation for the proof of Theorem~\ref{cH1}.

		\vspace{3pt}
		\noindent\bf{Theorem~B}\mdseries{} (Flatto-Odlyzko-Wales \cite{FOW}). \textit{For $d\geq5$, partitions $\lambda,\mu\vdash d$ with $\mu\neq(1^d)\vdash d$,
			\begin{align}
				\frac{|\chi_\lambda(\mu)|}{\chi_\lambda(1^d)}\leq1,
			\end{align}
			with equality if and only if $\lambda=(d)$ or $(1^d)$.}
		\vspace{3pt}
		
		\begin{lemma}\label{dH2}
			For any fixed $d\geq7,s\geq0$ and $\mu^{(1)},\dots,\mu^{(s)}\vdash d$, when $k\cdot l^*(\nu)+\sum_{i=1}^s l^*(\mu^{(i)})=\text{even}$,
			\begin{align}
				H^{X*}_{d}(\mu^{(1)},\dots,\mu^{(s)},\underbrace{\nu,\nu,\dots}_{k})= \frac{2d!^{2g(X)}}{d!^2}\prod_{i=1}^s\frac{d!}{z_{\mu^{(i)}}} &\sum_{1\leq m\leq \frac{d!}{z_{\nu}}}b_\nu^{X*}(\mu^{(1)},\dots,\mu^{(s)},m) m^{k}, \label{dHv1}
			\end{align}
			where $d!^{2g(X)}\prod_{i=1}^s\frac{d!}{z_{\mu^{(i)}}}b_\nu^{X*}(\mu^{(1)},\dots,\mu^{(s)},m)$ are integers with $b_\nu^{X*}(\mu^{(1)},\,\cdots,\,\mu^{(s)},\frac{d!}{z_{\nu}})=1$.
		\end{lemma}		
		\begin{proof}
			According to \eqref{Hurwitz1}, 
			\begin{align}
				H^{X*}_{d}(\mu^{(1)},\dots,\mu^{(s)},\underbrace{\nu,\nu,\dots}_{k})=\sum_{\lambda\vdash d}\big(\frac{d!}{\dim\lambda}\big)^{s+g(X)-2}(f_{(\nu)}(\lambda))^k\prod_{i=1}^s \frac{\chi_\lambda(\mu^{(i)})}{z_{\mu^{(i)}}} .\label{6}
			\end{align}				
			
			Since $\dim\lambda=\chi_\lambda(1^d)$ and \eqref{f}, we have the structure of disconnected Hurwitz numbers as follows: 
			\begin{align}
				H^{X*}_{d}(\mu^{(1)},\dots,\mu^{(s)},\underbrace{\nu,\nu,\dots}_{k})= \frac{2}{d!^2}\prod_{i=1}^s\frac{d!}{z_{\mu^{(i)}}} &\sum_{1\leq m\leq \frac{d!}{z_{\nu}}}b_\nu^{X*}(\mu^{(1)},\dots,\mu^{(s)},m) m^{k} ,\nonumber
			\end{align}
			where
			\begin{align}
				b^{X*}_\nu(\mu^{(1)},\dots,\mu^{(s)},m)=\frac{1}{2}\sum_{\substack{\lambda\vdash d\\|f_{\nu}(\lambda)|=m}}(\dim(\lambda))^{2-2g(X)}\big(\sgn(f_{\nu}(\lambda))\big)^k\prod_{i=1}^{s}\frac{\chi_\lambda(\mu^{(i)})}{\dim(\lambda)}.      \label{top2}
			\end{align}
			
			Notice that
			\begin{align}
				d!^{2g(X)}\prod_{i=1}^s\frac{d!}{z_{\mu^{(i)}}}b_\nu^{X*}(\mu^{(1)},\dots,\mu^{(s)},m)&=\frac{1}{2}\sum_{\substack{\lambda\vdash d\\|f_{\nu}(\lambda)|=m}}(\dim(\lambda))^{2}\Big(\frac{d!}{\dim(\lambda)}\Big)^{2g(X)}\nonumber\\
				&\times\big(\sgn(f_{\nu}(\lambda))\big)^k\prod_{i=1}^{s}f_{\mu^{(i)}}(\lambda). 
			\end{align}
			
			Since $f_\mu(\lambda)$ is an integer and $\chi_{\lambda}(\mu)=(-1)^{l^*(\mu)}\chi_{\lambda'}(\mu)$, we have $d!^{2g(X)}\prod_{i=1}^s\allowbreak \frac{d!}{z_{\mu^{(i)}}}b_\nu^{X*}(\mu^{(1)},\dots,\mu^{(s)},m)$ are integers.
			
			When $k\cdot l^*(\nu)+\sum_{i=1}^s l^*(\mu^{(i)})=\text{even}$, since Theorem B and $\frac{\chi_{(d)}(\mu)}{\chi_{(d)}(1^d)}=1,\frac{\chi_{(1^d)}(\mu)}{\chi_{(1^d)}(1^d)}=(-1)^{l^*(\mu)}$, we have $b_\nu^{X*}(\mu^{(1)},\,\cdots,\,\mu^{(s)},\frac{d!}{z_{\nu}})=1$.
			
		\end{proof}
		\begin{proof}[\bf{Proof of Theorem~\ref{cH1}}]			
			
			It follows from \eqref{cHdH}.
		\end{proof}

		\section{Further remarks}\label{Conclusion}
		We observe that Lemma~\ref{rm2} holds for $2\leq r\leq d$, while Theorem~\ref{cH2} only covers the range $2\leq r\leq d-2$. In this section, we apply Lemma~\ref{rm2} to the case $r=d-1$ and $r=d$, thereby obtaining Theorem~\ref{cH5} and Theorem~\ref{cH6}, respectively. Moreover, by replacing $r\,1^{d-r}$ with a general partition $\nu$, we propose Conjecture~\ref{conj1} as a stronger version of Lemma~\ref{rm2}. We then go on to propose Conjecture~\ref{cH4}, Conjecture~\ref{cH9} and Conjecture~\ref{cH11}. In probability theory, a series of uniform upper bounds on character radios that hold for all irreducible representations $\lambda$ have been established \cite{FS,LS,MS,R}, which may help to prove Conjecture~\ref{conj1}. We note this direction as a possible avenue for further investigation.		
		
		\begin{theorem}\label{cH5}
			For any fixed $d\geq7,s\geq0$ and $\mu^{(1)},\dots,\mu^{(s)}\vdash d$, we have
			\begin{align}
				&H^X_{g,d}(\mu^{(1)},\dots,\mu^{(s)},(d-1,1),(d-1,1),\dots)\nonumber\\ =&\frac{2d!^{2g(X)}}{d!^2}\prod_{i=1}^s\frac{d!}{z_{\mu^{(i)}}} \sum_{1\leq m\leq d(d-2)!}b^X_{d-1}(\mu^{(1)},\dots,\mu^{(s)},m) m^{\frac{1}{d-2}(2g+(2-2g(X))d-2-\sum_{i=1}^s l^*(\mu^{(i)}))}, \label{cHv3}
			\end{align}
			where $d!^{2g(X)}\prod_{i=1}^s\frac{d!}{z_{\mu^{(i)}}}b_{d-1}^{X}(\mu^{(1)},\dots,\mu^{(s)},m)$ are integers with
			
			1. $b^X_{d-1}(\mu^{(1)},\dots,\mu^{(s)},d(d-2)!)=1$;
			
			2. $b^X_{d-1}(\mu^{(1)},\,\cdots,\,\mu^{(s)},m)=0$ for $(d-2)!<m<d(d-2)!$;
			
			3. $b^X_{d-1}(\mu^{(1)},\,\cdots,\,\mu^{(s)},(d-2)!)=-d^{2-2g(X)-s}\prod_{i=1}^{s}m_1(\mu^{(i)})$;
			
			4. $b^X_{d-1}(\mu^{(1)},\dots,\mu^{(s)},m)=0$ for $2(d-2)(d-4)!<m<(d-2)!$;
			
			5. $b_{d-1}^{X}(\mu^{(1)},\,\cdots,\,\mu^{(s)},2(d-2)(d-4)!)=(d-1)^{2-2g(X)-s}\prod_{i=1}^{s}(m_1(\mu^{(i)})-1)$.		
		\end{theorem}	
		\begin{theorem}\label{cH6}
			For any fixed $d\geq7,s\geq0$ and $\mu^{(1)},\dots,\mu^{(s)}\vdash d$, we have
			\begin{align}
				&H^X_{g,d}(\mu^{(1)},\dots,\mu^{(s)},d,d,\dots)\nonumber\\ =&\frac{2d!^{2g(X)}}{d!^2}\prod_{i=1}^s\frac{d!}{z_{\mu^{(i)}}} \sum_{1\leq m\leq (d-1)!}b^X_{d}(\mu^{(1)},\dots,\mu^{(s)},m) m^{\frac{1}{d-1}(2g+(2-2g(X))d-2-\sum_{i=1}^s l^*(\mu^{(i)}))}, \label{cHv4}
			\end{align}
			where $d!^{2g(X)}\prod_{i=1}^s\frac{d!}{z_{\mu^{(i)}}}b_{d}^{X}(\mu^{(1)},\dots,\mu^{(s)},m)$ are integers with
			
			1. $b^X_{d}(\mu^{(1)},\dots,\mu^{(s)},(d-1)!)=1$;			
			
			2. $b^X_{d}(\mu^{(1)},\,\cdots,\,\mu^{(s)},m)=0$ for $(d-2)!<m<(d-1)!$;
			
			3. $b_{d}^{X}(\mu^{(1)},\,\cdots,\,\mu^{(s)},(d-2)!)=(d-1)^{2-2g(X)-s}\prod_{i=1}^{s}(m_1(\mu^{(i)})-1)$.		
		\end{theorem}	
		The proof of Theorem~\ref{cH5} and Theorem~\ref{cH6} follow the same lines as that of Theorem~\ref{cH2}, and are therefore omitted.
		
		Motivated by \cite{DYZ}, based on extensive numerical experiments, we propose the following conjecture.
		\begin{conjecture}\label{conj1}
			For $d\geq10$, let $\lambda,\mu\vdash d$ satisfy $\lambda\neq(d),(1^d)$. If $m_1(\mu)\neq1$ and $ \mu\neq(2^{d/2}), (2^{d/2-1},1^2)$, then
			\begin{align}
				\frac{|\chi_\lambda(\mu)|}{\chi_\lambda(1^d)}\leq&\frac{|m_1(\mu)-1|}{d-1},
			\end{align}
			with equality if and only if $\lambda=(d-1,1)$ or $(2,\,1^{d-2})$. If $m_1(\mu)=1,\,m_2(\mu)\neq1$ and $\mu\neq(3^{d/3-1},2,1)$, then
			\begin{align}
				\frac{|\chi_\lambda(\mu)|}{\chi_\lambda(1^d)}\leq&\frac{2m_2(\mu)}{(d-1)(d-2)},
			\end{align}
			with equality if and only if $\lambda=(d-2,1,1)$ or $(3,\,1^{d-3})$. If $m_1(\mu)=1,\,m_2(\mu)=0$ and $\mu\neq(3^{(d-1)/3},1)$, then
			\begin{align}
				\frac{|\chi_\lambda(\mu)|}{\chi_\lambda(1^d)}\leq&\frac{2}{d(d-3)},
			\end{align}
			with equality if and only if $\lambda=(d-2,2)$ or $(2,2,\,1^{d-4})$.
		\end{conjecture} 
		In Lemma~\ref{rm2} (see Section~\ref{proof}), using the result of Frumkin-James-Roichman \cite{FJR}, we establish Conjecture~\ref{conj1} for the case $\mu=(r,1^{d-r})$.
		A natural extension of Conjecture~\ref{conj1} leads us to the following 
		
		\begin{conjecture}\label{cH4}
			For $d\geq10,\,m_1(\nu)\geq2$, the coefficients $b^X_{\nu}(\mu^{(1)},\,\cdots,\,\mu^{(s)},m)$ in Theorem~\ref{cH1} equal to zero for $\frac{d!m_1(\nu)}{d\cdot z_\nu}<m<\frac{d!}{z_\nu}$ and $\frac{d!(m_1(\nu)-1)}{(d-1)z_\nu}<m<\frac{d!m_1(\nu)}{d\cdot z_\nu}$, and
			\begin{align}
				&b^X_{\nu}(\mu^{(1)},\,\cdots,\,\mu^{(s)},\frac{d!m_1(\nu)}{d\cdot z_\nu})=-d^{2-2g(X)-s}\prod_{i=1}^{s}m_1(\mu^{(i)}).
			\end{align}
			Moreover, for $\nu\neq(2^{d/2-1},1^2)$,
			\begin{align}
				b_{\nu}^{X}(\mu^{(1)},\,\cdots,\,\mu^{(s)},\frac{d!(m_1(\nu)-1)}{(d-1)z_\nu})=(d-1)^{2-2g(X)-s}\prod_{i=1}^{s}(m_1(\mu^{(i)})-1).
			\end{align}
		\end{conjecture}
		
		\begin{conjecture}\label{cH9}
			For $d\geq10,m_1(\nu)=0$ and $\nu\neq(2^{d/2})$, the coefficients $b^X_{\nu}(\mu^{(1)},\,\cdots,\,\mu^{(s)})$ in Theorem~\ref{cH1} equal to zero for $\frac{d(d-2)!}{z_\nu}<m<\frac{d!}{z_\nu}$ and
			\begin{align}
				b_{\nu}^{X}(\mu^{(1)},\,\cdots,\,\mu^{(s)},\frac{d(d-2)!}{z_\nu})=(d-1)^{2-2g(X)-s}\prod_{i=1}^{s}(m_1(\mu^{(i)})-1).
			\end{align}
		\end{conjecture}
		\begin{conjecture}\label{cH11}
			For $d\geq10,m_1(\nu)=1$ and $\nu\neq(3^{d/3-1},2,1)$ or $(2^{d/2-1},1)$ or $(3^{(d-1)/3},1)$, the coefficients $b^X_{\nu}(\mu^{(1)},\,\cdots,\,\mu^{(s)})$ in Theorem~\ref{cH1} equal to zero for $\frac{(d-1)!}{z_\nu}<m<\frac{d!}{z_\nu}$ and
			\begin{align}
				b_{\nu}^{X}(\mu^{(1)},\,\cdots,\,\mu^{(s)},\frac{(d-1)!}{z_\nu})=(d-1)^{2-2g(X)-s}\prod_{i=1}^{s}(m_1(\mu^{(i)})-1).
			\end{align}
			Moreover, the coefficients $b^X_{\nu}(\mu^{(1)},\,\cdots,\,\mu^{(s)})$ in Theorem~\ref{cH1} also equal to zero for $\frac{2dm_2(\nu)(d-3)!}{z_\nu}<m<\frac{(d-1)!}{z_\nu}$ when $m_2(\nu)\neq1$, and also equal to zero for $\frac{2(d-1)!}{z_\nu(d-3)}<m<\frac{(d-1)!}{z_\nu}$ when $m_2(\nu)=0$.
		\end{conjecture}

		\section*{Acknowledgement}
		I would like to thank Di Yang for his advice. This work was supported by NSFC No. 12371254 and CAS NO. YSBR-032.

	\end{CJK}

\begin{thebibliography}{99}
			
			\bibitem{ALR}
			\newblock D. Accadia, D. Lewanski and G. Ruzza.
			\newblock "On the large genus of Hurwitz numbers,"
			\newblock \emph{in preparation}.
			
			\bibitem{Burn}
			\newblock W. Burnside.
			\newblock "Theory of Groups of Finite Order,"
			\newblock \emph{2nd edition, Cambridge University Press}, 1911.
			
			\bibitem{DLLY}
			\newblock Y. Ding, K. Li, H. Liu and D. Yan.
			\newblock "The uniform asymptotic behavior for real double Hurwitz numbers with triple ramification,"
			\newblock \emph{arXiv}: 2303.03671.
			
			\bibitem{DLLY2}
			\newblock Y. Ding, K. Li, H. Liu and D. Yan.
			\newblock "The uniform asymptotics for real double Hurwitz numbers with triple ramification II: lower bounds and asymptotics,"
			\newblock \emph{arXiv}: 2602.04191.
			
			\bibitem{DHR}
			\newblock N. Do, J. He and H. Robertson.
			\newblock "The structure of Hurwitz numbers with fixed ramification profile and varying genus,"
			\newblock \emph{arXiv}: 2409.06655.
			
			\bibitem{DYZ}
			\newblock B. Dubrovin, D. Yang and D. Zagier.
			\newblock "Classical Hurwitz numbers and related combinatorics,"
			\newblock \emph{Moscow Mathematical Journal}, \textbf{17}, 601–633, 2017.
			
			\bibitem{FOW}
			\newblock L. Flatto, A.M. Odlyzko and D.B. Wales.
			\newblock "Random shuffles and group representations,"
			\newblock \emph{Ann. Prob}, \textbf{13}, 154–178, 1985.
			
			\bibitem{Fro}
			\newblock G. Frobenius.
			\newblock "\"Uber die Charaktere der symmetrischen gruppe,"
			\newblock \emph{Sitzber. Pruess. Akad. Berlin}, 516–534, 1900.
			
			\bibitem{FJR}
			\newblock A. Frumkin, G. James and Y. Roichman.
			\newblock "On trees and characters,"
			\newblock \emph{J. Algebraic Combin}, \textbf{17}, 323–334, 2003.						
			
			\bibitem{FS}
			\newblock V. F\'{e}ray and P. \'{S}niady.
			\newblock "Asymptotics of characters of symmetric groups related to Stanley character formula,"
			\newblock \emph{Ann. of Math}, \textbf{2}, 173, 887–906, 2011.
			
			\bibitem{Hur1}
			\newblock A. Hurwitz.
			\newblock "Ueber Riemann’sche Fl\"achen mit gegebenen Verzweigungspunkten,"
			\newblock \emph{Math. Ann}, \textbf{39(1)}, 1–60, 1891.
			
			\bibitem{Hur2}
			\newblock A. Hurwitz.
			\newblock "Ueber die Anzahl der Riemann’schen Fl\"achen mit gegebenen Verzweigungspunkten,"
			\newblock \emph{Math. Ann}, \textbf{55}, 53–66, 1901.
			
			\bibitem{L1}
			\newblock X. Li.
			\newblock "Combinatorics and asymptotic behavior for double Hurwitz numbers,"
			\newblock \emph{Preprint}.
			
			\bibitem{L2}
			\newblock X. Li.
			\newblock "On large genus asymptotics of certain Hurwitz numbers,"
			\newblock \emph{Preprint}.
			
			\bibitem{LS}
			\newblock M. Larsen and A. Shalev.
			\newblock "Characters of symmetric groups: sharp bounds and applications,"
			\newblock \emph{Invent. Math}, \textbf{3}, 174, 645–687, 2008.
			
			\bibitem{Mac}
			\newblock I. Macdonald.
			\newblock "Symmetric functions and Hall polynomials,"
			\newblock \emph{Oxford university press}, 0998.
			ISBN: 978-0-19-873912-8.			
			
			\bibitem{MS}
			\newblock T. W. M\"{u}ller and J-C. Schlage-Puchta.
			\newblock "Character theory of symmetric groups, subgroup growth of Fuchsian groups, and random walks."
			\newblock \emph{Adv. Math}, \textbf{2}, 213, 919–982, 2007.
			
			\bibitem{R}
			\newblock Y. Roichman.
			\newblock "Upper bound on the characters of the symmetric groups,"
			\newblock \emph{Invent. Math}, \textbf{125}, 451–485, 1996.
			
			\bibitem{Y}
			\newblock C. Yang.
			\newblock "The structures of simple Hurwitz numbers and monotone Hurwitz numbers with varying genus,"
			\newblock \emph{arXiv}: 2503.01920.
			
			
			
			
			
		\end{thebibliography}
\end{document}